\documentclass[twoside,12pt,a4paper]{amsart}
\usepackage{amssymb}
\usepackage{mathrsfs}
\usepackage{amsmath}
 \usepackage[colorlinks,citecolor=blue,urlcolor=blue, linkcolor=blue, backref]{hyperref}
 \usepackage{undertilde}
 \usepackage[capitalise]{cleveref}		
 \usepackage{upgreek}

\newtheorem{theorem}{Theorem}[section]
\newtheorem*{theorem*}{Theorem}

\newtheorem{corollary}[theorem]{Corollary}
\newtheorem{definition}[theorem]{Definition}

\newtheorem{lemma}[theorem]{Lemma}

\newtheorem{proposition}[theorem]{Proposition}
\newtheorem{question}[theorem]{Question}

\newcommand{\AC}{\mathrm{AC}}

\newcommand{\bi}{\begin{itemize}}
\newcommand{\ei}{\end{itemize}}
\newcommand{\bc}{\begin{center}}
\newcommand{\ec}{\end{center}}

\newcommand{\D}{\mathcal{D}}
\newcommand{\R}{\mathbb{R}}

\newcommand{\ZF}{{\mathrm{ZF}}}
\newcommand{\TD}{{\mathrm{TD}}}
\newcommand{\AD}{{\mathrm{AD}}}
 \newcommand{\DCR}{{\mathrm{DC}_{\mathbb{R}}}}
\newcommand{\CCR}{{\mathrm{CC}_{\mathbb{R}}}}
\newcommand{\x}{{\mathbf{x}}}
\newcommand{\y}{{\mathbf{y}}}
\newcommand{\z}{{\mathbf{z}}}
\newcommand{\ux}{{u_{\mathbf{x}}}}

\begin{document}
\title{$\TD$ implies $\CCR$}
 \author{Yinhe Peng and Liang Yu }
 \thanks{Yu was partially supported by NSF of China No. 11671196 and 12025103. Both authors would like to thank Paul Larson for his suggestions and comments for the paper.}
 \address{Institute of Mathematics, Chinese Academy of Sciences\\ 
 East Zhong Guan Cun Road No. 55\\Beijing 100190\\China}
 \email{pengyinhe@amss.ac.cn}
\address{Department of Mathematics\ \\
Nanjing University, Jiangsu Province 210093\\
P. R. of China} \email{yuliang.nju@gmail.com.}
 
\maketitle
\begin{abstract}
Assuming $\ZF$, we prove that Turing determinacy ($\TD$)  implies the countable choice axiom for sets of reals ($\CCR$). \end{abstract}
\section{Introduction} \label{sec:introduction}

Turing reduction $\leq_T$ is a partial order over reals. It naturally induces an equivalence relation $\equiv_T$. Given a real $x$, its corresponding  Turing degree $\x$ is a set of reals defined by $\{y\mid y\equiv_T x\}$. We say $\x\leq \mathbf{y}$ if $x\leq_T y$.  We use $\D$ to denote the set of Turing degrees. An {\em upper cone} $u_{\x}$ of Turing degrees is the set $\{\mathbf{y}\mid \mathbf{y}\geq \mathbf{x}  \}$.

\begin{definition}
 Turing determinacy,  or $\TD$, says that for any set $A$ of Turing degrees, either $A$ or $\D\setminus A$ contains an upper cone of Turing degrees.
\end{definition}

Martin proves the following famous theorem.
\begin{theorem}[Martin \cite{Martin68}] 
Over $\ZF$, the Axiom of determinacy, or $\AD$,  implies $\TD$.
\end{theorem}

\begin{definition}
\begin{itemize}
 \item Countable choice axiom for sets of reals,  or $\CCR$, says that for any countable sequence $\{A_n\}_{n\in \omega}$ of nonempty sets of reals, there is a function $f:\omega\to \mathbb{R}$ so that for every $n$, $f(n)\in A_n$.
 \item  Dependent choice axiom for sets of reals,  or $\DCR$, says that for any binary relation $R$ over reals so that $\forall x \exists y R(x,y)$, there is a function $f:\omega\to \mathbb{R}$ so that for every $n$, $R(f(n), f(n+1))$.
 \end{itemize}
\end{definition}

Though $\AD$ contradict the Axiom of choice, or $\AC$, Mycielski proves the following theorem.
\begin{theorem}[Mycielski \cite{MYC64}]
Over $\ZF$,  $\AD$ implies $\CCR$.
\end{theorem}

It is a long standing question whether $\AD$ implies $\DCR$. 

The following question has been also circulated among set theorists (for example, see \cite{Lar}).
\begin{question}\label{question: td ccr?}
Over $\ZF$, does $\TD$ imply $\CCR$?
\end{question}

 In this paper, we answer the question.
 
 \bigskip
 
 We assume that readers have some knowledge of   descriptive set theory and recursion theory. The major references are \cite{Jech03} and \cite{Ler83}.

\section{The main theorem}

Throughout the section, we assume $\ZF+\TD$. 

We identify $\mathbb{R}$ as $2^{\omega}$ or $\mathscr{P}(\mathbb{N})$, the power set of $\mathbb{N}$. So if $x\in \R$, then $x(n)$, the $n$-th bit of $x$, belongs to $\{0,1\}$.  The structure of the Turing degrees is an upper semi-lattice. I.e. for any reals $x$ and $y$, $x\oplus y=\{2n\mid  x(n)=1\wedge n\in \mathbb{N} \}\wedge \{2n+1\mid y(n)=1 \wedge n\in \mathbb{N}\}$ has the least Turing degree above both $\x$ and $\y$.

\bigskip

We prove the following theorem.
\begin{theorem}\label{main}
Over $\ZF$,  $\TD$ implies $\CCR$
\end{theorem}

The following reformulation   of $\CCR$ is helpful to understand the idea behind the subsequent proofs.
\begin{proposition}\label{proposition: ccr and function}
$\CCR$ is equivalent to that for any function  $f:\D\to \R$, there is an upper cone on which $f$ is constant.
\end{proposition}
\begin{proof}
Assume $\CCR$. For any $i$, there is a  $j_i\in \mathbb{N}$ so that $A_i=\{\x\mid f(\mathbf{x})(i)=j_i\}$ contains an upper cone.  By  $\CCR$, we can choose, for each $i\in\mathbb{N}$, $\mathbf{z}_i$ such that $u_{\mathbf{z}_i}\subset A_i$ and  some $\mathbf{z}$ above all $\mathbf{z}_i$. Then $f$ is constant on $u_{\z}$.

Now suppose that $\{A_n\}_{n\in \omega}$ is a sequence of nonempty sets of reals witnessing the failure of $\CCR$. For any degree $\mathbf{x}$, let $i_{\mathbf{x}}$ be the least $i$ so that there is no real in $A_i$ Turing below $\mathbf{x}$. By the assumption, $i_{\x}$ exists for any degree $\x $.  Define a function $f:\D\to \R$ as follows:

\begin{equation*}
f(\mathbf{x})(n)=\left\{
\begin{aligned}
0, & \  & n<i_{\mathbf{x}}; \\
1, & \  & otherwise. \\
\end{aligned}
\right.
\end{equation*}

By the assumption, $f$ is well defined for every degree.  But clearly $f$ cannot be constant on any upper cone, a contradiction.
\end{proof}

The following technical lemma, which can be proved with $\ZF$,  is folklore in recursion theory.

\begin{lemma}\label{lemma: doubjump}
For any degree $\x$, there is a family Turing degrees $\{\y_r\mid r\in \mathbb{R}\}$ satisfying the following property:
\begin{itemize}
\item[(1)] For any $r\in \R$, $\x<\y_r$; 
 \item[(2)] For any $r_0\neq r_1\in \R$ and $\z<\y_{r_0},\y_{r_1}$, we have that $\z\leq \x$;
\item[(3)] For any $\z\geq \x''$, the Turing double jump of $\x$, there is an infinite set $C_{\z}\subset \R$ so that $\y_r''=\z$ for any $r\in C_{\z}$.
\end{itemize}
\end{lemma}
\begin{proof}
Fix a real $x$. It is routine (see \cite{Ler83}) to prove, by a  recursive Sacks forcing relative to $x$, that there is  a perfect tree $S\subset 2^{<\omega}$ with $S\leq_T x''$ satisfying the following property:
\begin{itemize}
 \item[(a)] For any $ y\in [S]$, the collection of infinite paths through $S$, we have that $x<_T y$ and there is no real $z$ so that $x<_T z<_T y$; 
 \item[(b)] For any $y_0\neq y_1\in [S]$,  $y_0$ is Turing incomparable with $y_1$;
 \item[(c)] For any $ y\in [S]$, $y''\equiv_T y\oplus x''$.
\end{itemize}

Then for any reals $y_0\neq y_1\in [S]$, $z<_T y_0$ and $z<_T y_1$, we have that $x\leq_T z\oplus  x\leq_T   y_0$ and $x\leq_T z\oplus  x\leq_T   y_1$. Thus, by (a) and (b),  $x\equiv_T z\oplus x$ and so $z\leq_T x$. 

Also for any $z\geq_T x''$, it is clear that there is an infinite set $C_z\subset [S]$ so that for any $y\in C_z$, we have that  $y\oplus S\equiv_T z$.  Then for any $y\in C_z$,  $y''\equiv_T y\oplus x''\geq_T y\oplus S\equiv_T z$. Also $y''\equiv_T y\oplus x''\leq_T z$. So $y''\equiv_T z$.

Now let $f$ be a bijection between $\R$ and $[S]$ and set $\y_r$ to be the Turing degree of $f(r)$ for any $r\in \R$. Then the collection $\{\y_r\}_{r\in \R}$ is exactly  what we want.
\end{proof}

 Firstly, we prove the following weak form of $\CCR$.

 \begin{proposition}\label{proposition: ucc}
Suppose that $\{A_n\}_{n\in \omega}$ is a countable family of countable sets of reals, then $A=\bigcup_n A_n$ is countable.
 \end{proposition}
 \begin{proof}
 
 Suppose not. Let $\{A_n\}_{n\in \omega}$ is a countable family of countable sets of reals so that $A=\bigcup_n A_n$ is not countable. Given a  real $x$, let $n_{\x}$ be the least number $n$ so that there is a real $z\in A_n$ that is not Turing below $x$. Then $n_{\x}$ is defined for every $\x$. Moreover, $n_{\x}\leq n_{\y}$ for any $\x\leq \y$.

 Now fix a degree $\x$. Let $\{\y_r\mid r\in \mathbb{R}\}$  be as in  Lemma \ref{lemma: doubjump}.  
Since $A_{n_{\x}}$ is countable, by (3) of the lemma, there is an uncountable set $Z\subseteq \R$ so that for any $r\in Z$, $n_{\y_r''}>n_{\x}$.  
 
 We claim that there must be some $r\in Z$ so that $n_{\y_r}=n_{\x}$ and hence $n_{\y_r}< n_{\y_r''}$. Otherwise, for any $r\in Z$, there is some $s\in A_{n_{\x}}$ so that $\bf{s}\leq \y_r$ but $\bf{s}\not\leq \bf{x}$.  Then by (1) and (2)   of the Lemma, $A_{n_{\x}}$ must be uncountable, which is  a contradiction to the assumption.
 
 Then by $\TD$, there is some $\x_0$ so that for any $\y\in u_{\x_0}$, $n_{\y}<n_{\y''}$.  Then $n_{(\x_0)^{(\omega)}}$ is undefined where $(\x_0)^{(\omega)}$ is the $\omega$-th Turing jump of $\x_0$. A contradiction.     
    \end{proof}

By Proposition \ref{proposition: ucc}, we have the following conclusion. 
\begin{corollary}\label{corollary: upper bound of countable turing degrees}
Every countable set of Turing degrees has an upper bound.
\end{corollary}

\begin{definition}
\begin{itemize}
\item A function $\Phi:\D\to \D$ is called {\em almost increasing}, if there is an upper cone $\ux$ so that for any $\mathbf{y}$ in $\ux$, $\y<\Phi(\mathbf{y})$. 

\item Fix a function $\Phi:\D\to \D$. We say that a function $f:\D \to \R$ is {\em almost injective corresponding to $\Phi$} if  there is an upper cone $\ux$ so that for any $\mathbf{y}\in \ux$, $f(\mathbf{y})\neq f(\Phi(\mathbf{y}))$. 
\end{itemize}
\end{definition}

For example, the Turing jump function $J:\x\mapsto \x'$ is almost increasing.

\begin{lemma}\label{ord lem}
Every function $F: \D\to Ord$ is non-decreasing  over an upper cone.
\end{lemma}
\begin{proof}
It suffices to prove that
$$L=\{\mathbf{x}\mid \forall \y\geq \x (F(\x)\leq F(\y)\}$$ 
is cofinal and hence contains an upper cone. To see this, just note that for any $\x$, there is a $\y\geq \x$ such that 
$$F(\y)=\min \{F(\z): \z\geq \x\}.$$
It is clear that $\y\in L$.
\end{proof}

\begin{lemma}\label{prop: no a increasing}
For any almost increasing function $\Phi$, there is no  almost injective function corresponding to it.
\end{lemma}
\begin{proof}
Suppose not.  Fix an almost increasing function $\Phi:\D \to \D$ witnessed by an upper cone $\ux$ and an  almost injective function $f$ corresponding to it. For $\y,\z \in \ux$, define $$l(\mathbf{y},\mathbf{z})=\min\{l\mid f(\mathbf{y})(l)\neq f(\mathbf{z})(l)\}.$$ Note that, by the assumption, $l(\mathbf{y},\mathbf{\Phi(\y)})$ is defined for every $\y\in \ux$. 

By Lemma \ref{ord lem} $$L_1=\{\mathbf{y}\mid \forall \z\geq \y (l(\mathbf{y},\Phi(\y))\leq l(\z,\Phi(\z)))\}$$ contains an upper cone $u_{\x_1}$ for some $\x_1\geq\x$.

Also for $i\in \{0,1\}$, let $$L_2^i=\{\mathbf{y}\mid f(\mathbf{y})(l(\mathbf{y},\Phi(\y))=i)\}.$$ Then for some $i$,   $L_2^i$ contains an upper cone $\mathbf{x}_2\geq \mathbf{x}_1$. Then for any $\mathbf{y}\geq \mathbf{x}_2$, $ f(\mathbf{y})(l(\mathbf{y},\Phi(\y)))=i= f(\Phi(\y))(l(\Phi(\y),\Phi(\Phi(\y))))$. So by the definition of $l$,  $$ l(\mathbf{y},\Phi(\y))<l(\Phi(\y),\Phi(\Phi(\y))).$$

By Corollary \ref{corollary: upper bound of countable turing degrees}, there is a degree $\z$ which is above $\Phi^{(n)}(\y)$ for every $n$. Then $l(\z,\Phi(\z))$ is undefined since $\z>\x_1$.
\end{proof}

\begin{lemma}\label{lemma: down}
Suppose that $f:\D\to \R$ is a function so that  $f(\mathbf{y})\leq \mathbf{y}$\footnote{Here we identify a real $y$ with its degree $\y$.} over an upper cone, then the range of $f$ is at most countable over an upper cone.
\end{lemma}
\begin{proof}
Suppose not. 

By Lemma \ref{prop: no a increasing}  and applying $\x\mapsto \x'$ to $\Phi$, we get that $f(\mathbf{y})=f(\mathbf{y}')$ over an upper cone. In particular,  $f(\mathbf{y})\notin \mathbf{y}$.
 
So we can find an upper cone $\ux$ so that  $f(\y')=f(\mathbf{y})< \mathbf{y}$ for any $\y\in\ux$. We fix a set $\{\y_r\mid r\in \R\}$ as in Lemma \ref{lemma: doubjump}. By applying the assumption to the upper cone $u_{\x''}$, we may pick a  real $r_0$  with $\y_{r_0}''\geq \x''$ so that   $f(\mathbf{y}_{r_0}'')\not\leq\mathbf{x}$. By (3) of Lemma \ref{lemma: doubjump}, there is another real $r_1\neq r_0$ so that $\y_{r_0}''=\y_{r_1}''$ and so $f(\mathbf{y}_{r_0}'')=f(\mathbf{y}_{r_1}'')$. 

However, $f(\mathbf{y}_{r_i}'')=f(\mathbf{y}_{r_i}')=f(\mathbf{y}_{r_i})\leq \mathbf{y}_{r_i}$ for both $i\leq 1$. By (2) of Lemma \ref{lemma: doubjump},  $f(\mathbf{y}_{r_i}'')\leq \x$ for both $i\leq 1$, contradicting the selection of $\y_{r_0}$.  
\end{proof}

\begin{lemma}\label{proposition: countable}
Suppose that $f:\D\to \R$ is a function. The range of $f$ is at most countable over an upper cone.
\end{lemma}
\begin{proof}
Suppose not.  Then by Lemma \ref{lemma: down}, we may assume that over an upper cone $\ux$, 
 $$f(\mathbf{y})\nleq \y.\eqno(*)_{\y}$$ 
 Let $$\Phi(\mathbf{y})=\{z\mid \exists y_0\in \y(z\equiv_T y_0\oplus f(\y)) \}.$$ Then $\Phi$ is a well defined almost increasing function over the upper cone $\ux$. Then by Lemma \ref{prop: no a increasing},   $f(\mathbf{y})= f(\Phi(\y))$ over an upper cone $u_{\x_1}\subseteq \ux$. Then over the upper cone, $$ f(\Phi(\y))=f(\mathbf{y}) \leq \Phi(\y).$$
This contradicts  $(*)_{\Phi(\y)}$.
\end{proof}

Now we are ready to prove the main theorem \ref{main}.

\begin{proof}(of Theorem \ref{main}.)
Suppose that $\{A_n\}_{n\in \mathbb{N}}$ is a countable family of nonempty sets of reals. Without loss of generality, we may assume that $A_n$ is Turing upward closed. Furthermore, we also may assume that $A_{n+1 } \subset A_{n}$ for every $n$ (reset $A_{n}$ to be $\bigcap_{k\leq n}A_k$ if necessary). 

For a contradiction, we  assume that for any real $y$, there is some $n$ (and so there are infinitely many $n$'s) so that there is  no real in $A_n$ Turing below $y$.  Define $B_n=A_n\setminus A_{n+1}$ for every $n$. Note that $B_n$ is nonempty and disjoint from an upper cone for every $n$.

Define a function $f:\D\to \R$\footnote{Here we identify a real as a subset of natural numbers.} so that $$f(\mathbf{y})=\{n\mid \exists \mathbf{z}\geq \mathbf{y}(\mathbf{z}\in B_n)\}.$$ By the property of $B_n$'s,  for every $\y$, $f(\y)(n)=1$ for infinitely many $n$'s.     By Lemma \ref{proposition: countable}, $f$ is countable over an upper cone of Turing degrees. Let $\{a_i\}_{i\in \omega}$ be an enumeration of the range of $f$ over the upper cone. Note that for every $i$,   $a_i(n)=1$ for infinitely many $n$'s  . Then there is some $a\subseteq \omega$ so that $a\cap a_i\neq \emptyset$ and $a_i\setminus a\neq \emptyset$ for every $i$. Let $C_0=\bigcup_{n\in a}B_n$ and $C_1=\bigcup_{n\not\in a}B_n $. Then $C_0\cap C_1=\emptyset$ and $C_0\cup C_1=\bigcup_{n\in \mathbb{N}} B_n$. So either $C_0$ or $C_1$ contains an upper cone of Turing degrees.\footnote{Actually ranges of the $f$ over upper cones   generate an ultrafilter as observed by Larson \cite{Lartd}.}

If $C_0$ contains an upper cone $\ux$ of Turing degrees, then let $\mathbf{y}\in \ux$. Then for any $\mathbf{y}_0\geq\y$, $\mathbf{y}_0\not\in C_1$ and so  $f(\mathbf{y})\subseteq a$. But $a_i\not\subseteq a$ for every $i$. Thus $f(\mathbf{y})\neq a_i$ for every $i$. 

If $C_1$ contains an upper cone $\ux$ of Turing degrees, then let $\mathbf{y}\in \ux$.  Then for any $\mathbf{y}_0\geq\y$, $\mathbf{y}_0\not\in C_0$ and   so  $f(\mathbf{y})\cap a=\emptyset$. But $a_i\cap a\neq \emptyset$ for every $i$. Thus $f(\mathbf{y})\neq a_i$ for every $i$. 

So in either case, there is some  $\y$ so that $f(\y)$ is not in the range of $f$, which is absurd.
\end{proof}

\bibliographystyle{plain}

\end{document}